\documentclass[11pt]{amsart}

\usepackage{amsmath,amssymb,amsthm}
\usepackage[margin=1.15in]{geometry}
\usepackage{enumitem}
\usepackage[pdfusetitle,colorlinks=true,linkcolor=blue,citecolor=blue,urlcolor=blue]{hyperref}

\newtheorem{theorem}{Theorem}[section]
\newtheorem{lemma}[theorem]{Lemma}
\newtheorem{proposition}[theorem]{Proposition}
\newtheorem{corollary}[theorem]{Corollary}
\newtheorem*{mainthm}{Theorem A}
\newtheorem*{mainthmB}{Theorem B}
\theoremstyle{definition}
\newtheorem{definition}[theorem]{Definition}
\newtheorem{question}[theorem]{Question}
\theoremstyle{remark}
\newtheorem{remark}[theorem]{Remark}

\newcommand{\K}{\mathbb{K}}
\newcommand{\R}{\mathbb{R}}
\newcommand{\C}{\mathbb{C}}
\newcommand{\N}{\mathbb{N}}
\newcommand{\Z}{\mathbb{Z}}
\newcommand{\Q}{\mathbb{Q}}
\newcommand{\M}{\mathrm{M}}
\newcommand{\GL}{\mathrm{GL}}
\newcommand{\SL}{\mathrm{SL}}
\newcommand{\SO}{\mathrm{SO}}
\newcommand{\cI}{\mathcal{I}}
\newcommand{\cS}{\mathcal{S}}
\newcommand{\cG}{\mathcal{G}}
\newcommand{\cC}{\mathcal{C}}
\newcommand{\cR}{\mathcal{R}}
\newcommand{\Ups}{\Upsilon}
\newcommand{\bUps}{\overline{\Upsilon}}
\newcommand{\sgn}{\operatorname{sgn}}
\newcommand{\diag}{\operatorname{diag}}
\newcommand{\tr}{\operatorname{tr}}
\newcommand{\rank}{\operatorname{rank}}
\newcommand{\gen}[1]{\langle #1\rangle}
\newcommand{\blk}[4]{\begin{pmatrix} #1 & #2\\ #3 & #4\end{pmatrix}}

\makeatletter
\@namedef{subjclassname@2020}{\textup{2020} Mathematics Subject Classification}
\makeatother

\begin{document}

\title{Maximally transitive semigroups of $n\times n$ matrices \\ (revised and extended version)}

\author{Mohammad Javaheri}
\address{Department of Mathematics and Statistics, Siena University\\
515 Loudon Road, Loudonville, NY 12211, USA}
\email{mjavaheri@siena.edu}

\date{Revised version, October 2026}

\subjclass[2020]{Primary 47D03, 20M20; Secondary 47A16, 37C85}
\keywords{Dense subsemigroups, matrix semigroups, hypercyclicity, topological transitivity}

\begin{abstract}
Let $\K=\R$ or $\C$. We prove that for every $n\ge 1$ there is a pair of
$n\times n$ matrices with entries in $\K$ that generates a dense
subsemigroup of the space $\M_n(\K)$ of all $n\times n$ matrices. Since the
powers of a single matrix are never dense, two generators is optimal.
Equivalently, there is a $2$-generator semigroup of linear maps whose action
on $(\K^n)^n$ is topologically transitive, whereas no semigroup of linear
maps acts topologically transitively on $(\K^n)^{n+1}$. In fact, for such a
semigroup every orbit closure in $(\K^n)^k$, for every $k\ge1$, is as large
as linear algebra allows. Along the way we give a short self-contained proof
of the density of explicit $2$-generator subsemigroups of $\GL_2^+(\R)$ and
of $\M_2(\R)$, based on a simple criterion for closed subsemigroups of
$\SL_2(\R)$.
\end{abstract}

\maketitle

\section{Introduction}

\subsection*{Kronecker's approximation theorem}
The one-dimensional version of Kronecker's approximation theorem
\cite{Apostol} states that, given an irrational number $\theta$, a real
number $r$ and $\epsilon>0$, there exist integers $m$ and $k$ such that
$|r-m\theta-k|<\epsilon$. In other words, $\{m\theta+k: m,k\in\Z\}$ is dense
in $\R$ if (and only if) $\theta$ is irrational. From this one deduces (see
Lemma~\ref{lem:dim1}) that if $a<0<b$ and $\ln|a|/\ln b$ is irrational, then
the semigroup
\[
\gen{a,b}=\{a^mb^k: m,k\ge0,\ m+k\ge1\}
\]
is dense in $\R$. In this paper we are interested in the following
$n$-dimensional generalization of this observation.

\begin{question}
Let $\K=\C$ or $\R$. Does there exist a pair $(A,B)$ of $n\times n$
matrices with entries in $\K$ such that the semigroup
\[
\gen{A,B}=\{A^{m_1}B^{k_1}\cdots A^{m_r}B^{k_r}:\ r\ge1,\ m_i,k_i\ge0,\
\textstyle\sum_i(m_i+k_i)\ge1\}
\]
generated by $A$ and $B$ is dense in the set $\M_n(\K)$ of all $n\times n$
matrices?
\end{question}

Our first main result answers this question in the affirmative.

\begin{mainthm}
For $\K=\R$ or $\C$ and every $n\ge1$ there exist $A,B\in\M_n(\K)$ such that
$\gen{A,B}$ is dense in $\M_n(\K)$. Necessarily $A,B\in\GL_n(\K)$, so
$\gen{A,B}$ is also a dense subsemigroup of $\GL_n(\K)$.
\end{mainthm}

Theorem~A is the conjunction of Theorems~\ref{thm:complex} and
\ref{thm:real} together with Lemma~\ref{lem:invertible}. The proof is by
induction on $n$, and the inductive step (Proposition~\ref{prop:extension})
is a general principle that may be of independent interest: a closed
subsemigroup of $\M_{n+1}(\K)$ that contains all matrices
$\diag(M,1)$, $M\in\GL_n(\K)$, together with one suitable matrix, is all of
$\M_{n+1}(\K)$.

\subsection*{Hypercyclicity}
Given an action of a semigroup $G$ on a topological space $X$, the action is
called \emph{hypercyclic} if there exists $x\in X$ whose orbit
$\{g(x): g\in G\}$ is dense in $X$. Feldman \cite{Feldman} proved that there
exist hypercyclic abelian semigroups of $\C^n$ generated by $n+1$ diagonal
matrices, and that no $n$-tuple of diagonalizable matrices is hypercyclic.
Costakis, Hadjiloucas and Manoussos \cite{CHM} showed that, once
diagonalizability is dropped, there are hypercyclic abelian semigroups of
$\R^n$ generated by $n$ matrices. The minimal number of generators of a
hypercyclic finitely generated abelian semigroup of matrices was then
determined: it is $n+1$ on $\C^n$ \cite{Ayadi,AbelsManoussos} and
$\lfloor (n+1)/2\rfloor+1$ on $\R^n$ \cite{AyadiMarzougui,AbelsManoussos}.

In the non-abelian case it was shown in \cite{J-orbits} that there exist
$2$-generator hypercyclic semigroups of $\K^n$ in every dimension, in both
the real and the complex case. Theorem~A is much stronger: a dense
subsemigroup of $\M_n(\K)$ moves every nonzero vector to a dense set of
vectors, and in fact moves every linearly independent $n$-tuple of vectors
to a dense set of $n$-tuples. Since the powers of a single matrix are never
dense \cite{Rolewicz}, two generators is optimal.

\subsection*{Topological $k$-transitivity}
The action of a semigroup $G$ on a topological space $X$ is
\emph{topologically transitive} if for every pair of nonempty open sets
$U,V\subseteq X$ there exists $g\in G$ with $g(U)\cap V\neq\emptyset$. It is
\emph{topologically $k$-transitive} if the diagonal action of $G$ on $X^k$
is topologically transitive. Ayadi \cite{Ayadi} proved that the action of a
finitely generated abelian semigroup of matrices on $\R^n$ or $\C^n$ is never
topologically $k$-transitive for $k\ge2$. In the non-abelian case we prove
the following, which explains the title.

\begin{mainthmB}
Let $\K=\R$ or $\C$ and $n\ge1$.
\begin{enumerate}[label=\rm(\roman*)]
\item No semigroup of linear maps of $\K^n$ acts topologically
$(n+1)$-transitively on $\K^n$
(Proposition~\ref{prop:not-n+1}).
\item For a subsemigroup $G\subseteq\M_n(\K)$, being dense in $\M_n(\K)$,
acting topologically $n$-transitively on $\K^n$, and having a dense orbit in
$(\K^n)^n$ are equivalent (Proposition~\ref{prop:equiv}). In particular,
there are $2$-generator semigroups acting topologically $n$-transitively.
\item If $G$ is dense in $\M_n(\K)$, then for every $k\ge1$ and every
$X=(X_1,\dots,X_k)\in(\K^n)^k$, the closure of the $G$-orbit of $X$ is the
linear subspace of all $(Y_1,\dots,Y_k)$ satisfying every linear relation
$\sum_i c_iX_i=0$ satisfied by $X$; this is the largest set that can be the
orbit closure of $X$ under any semigroup of linear maps
(Theorem~\ref{thm:orbit-closures}).
\end{enumerate}
\end{mainthmB}

\subsection*{Organization}
Section~\ref{sec:prelim} contains the block-matrix machinery and the
extension principle (Proposition~\ref{prop:extension}).
Section~\ref{sec:complex} treats the complex case and
Section~\ref{sec:real} the real case; the real case is self-contained apart
from Kronecker's theorem. Section~\ref{sec:transitive} contains the
transitivity results and Section~\ref{sec:remarks} some remarks and open
questions.

\section{Preliminaries}\label{sec:prelim}

Throughout, $\K=\R$ or $\C$, $\K^*=\K\setminus\{0\}$, $\M_{n\times k}(\K)$
is the space of $n\times k$ matrices over $\K$, $\M_n(\K)=\M_{n\times n}(\K)$,
and $\GL_n(\K)$ is the group of invertible matrices. We write
$\GL_n^\pm(\R)=\{M\in\GL_n(\R):\pm\det M>0\}$. Vectors are columns, $M^T$ is
the transpose of $M$, $I_n$ is the identity, and $e_1,\dots,e_n$ is the
standard basis of $\K^n$. For a subset $S$ of a semigroup, $\gen{S}$ denotes
the set of all finite (nonempty) products of elements of $S$. The closure of
a subsemigroup of $\M_n(\K)$ is again a subsemigroup, since multiplication is
continuous. Density of a semigroup is unaffected by a simultaneous change of
basis of its generators, a fact we use repeatedly without comment.

\begin{lemma}\label{lem:invertible}
If $A,B\in\M_n(\K)$ and $\gen{A,B}$ is dense in $\M_n(\K)$, then
$A,B\in\GL_n(\K)$ and $\gen{A,B}$ is dense in $\GL_n(\K)$.
\end{lemma}

\begin{proof}
Suppose $\det B=0$. Every product involving $B$ is singular, so
$\gen{A,B}\cap\GL_n(\K)\subseteq\{A^m:m\ge1\}$. Since $\GL_n(\K)$ is open,
$\{A^m\}$ would be dense in $\GL_n(\K)$, and then for any $v\neq0$ the set
$\{A^mv\}$ would be dense in $\GL_n(\K)v=\K^n\setminus\{0\}$, contradicting the
fact that no operator on a finite-dimensional space is hypercyclic
\cite{Rolewicz}. By symmetry $A$ is invertible too. The last statement holds
because $\GL_n(\K)$ is open.
\end{proof}

\begin{lemma}\label{lem:dim1}
\begin{enumerate}[label=\rm(\alph*)]
\item If $a<0<b$ and $\ln|a|/\ln b\notin\Q$, then $\gen{a,b}$ is dense in $\R$.
\item If $\beta>0$ and $\theta\in\R$ are such that $1,\beta,\theta$ are
linearly independent over $\Q$, then for $a=e^{-\beta+2\pi i\theta}$ and
$b=e$ the semigroup $\gen{a,b}$ is dense in $\C$.
\end{enumerate}
\end{lemma}

\begin{proof}
(a) Write $\alpha=\ln|a|$, $\gamma=\ln b$; they are nonzero, have opposite
signs and $\alpha/\gamma\notin\Q$. Then $\{2m\alpha+k\gamma:m,k\ge0\}$ is
dense in $\R$ by Kronecker's theorem, hence so is
$\alpha+\{2m\alpha+k\gamma\}$. Thus $\{a^{2m}b^k\}$ is dense in $(0,\infty)$
and $\{a^{2m+1}b^k\}$ is dense in $(-\infty,0)$.

(b) Let $w=e^{t+2\pi i\varphi}\neq0$ and $\eta>0$. By Kronecker's theorem
\cite[Thm.~7.10]{Apostol}, the points $(m\beta,m\theta)$, $m\in\N$, are
dense modulo $\Z^2$, so there are arbitrarily large $m$ with
$\|m\beta+t\|<\eta$ and $\|m\theta-\varphi\|<\eta$, where $\|\cdot\|$ is the
distance to the nearest integer. For $m>|t|/\beta$ let $k\ge0$ be the integer
nearest to $t+m\beta$. Then $a^mb^k=e^{(k-m\beta)+2\pi i m\theta}$ is within
$O(\eta)$ of $w$. Finally $a^m\to0$.
\end{proof}

\subsection{Block matrices}
For $n\ge1$ we write $(n+1)\times(n+1)$ matrices in block form
\begin{equation}\label{eq:block}
G=\blk{F}{X}{Y^T}{\eta},\qquad F\in\M_n(\K),\ X,Y\in\K^n,\ \eta\in\K .
\end{equation}
Let $\cG_n(\K)$ be the set of such $G$ with $F$ invertible; it is open and
dense in $\M_{n+1}(\K)$. Let
\[
\cI_n(\K)=\{G\in\cG_n(\K):\ Y^TF^{-1}X\neq0\},\qquad
\cI_n^\pm(\R)=\{G\in\cI_n(\R):\ \pm\det F>0\},
\]
\[
\cS_n(\K)=\Big\{\blk{F}{0}{0}{1}: F\in\GL_n(\K)\Big\},\qquad
\cS_n^\pm(\R)=\Big\{\blk{F}{0}{0}{1}: F\in\GL_n^\pm(\R)\Big\}.
\]
(In the original version $\cS_n(\K)$ was erroneously said to be contained
in $\cI_n(\K)$; it is not, since $Y^TF^{-1}X=0$ for its elements.) Define
\begin{equation}\label{eq:Ups}
\bUps:\cG_n(\K)\to\K^2,\qquad \bUps(G)=\big(Y^TF^{-1}X,\ \eta\big),
\end{equation}
and let $\Ups=\bUps|_{\cI_n(\K)}$, so that $\Ups$ takes values in
$\K^*\times\K$.

\begin{lemma}\label{lem:dual}
Let $n\ge1$ and $P_0,Q_0,P,Q\in\K^n\setminus\{0\}$ with $Q_0^TP_0=Q^TP\neq0$.
Then there is $M\in\GL_n(\K)$ with
\begin{equation}\label{eq:dual}
M^{-1}P_0=P,\qquad M^TQ_0=Q .
\end{equation}
If $\K=\R$ and $n\ge2$, then $M$ can be chosen with $\det M>0$.
\end{lemma}

\begin{proof}
First let $P_0=Q_0=e_1$, so $Q^TP=1$. Then $\K^n=\K P\oplus\ker Q^T$; let
$P_2,\dots,P_n$ be a basis of $\ker Q^T$ and $N=[P\,|\,P_2|\cdots|P_n]$. The
first row $r_1$ of $N^{-1}$ satisfies $r_1P=1$ and $r_1P_j=0$ for $j\ge2$,
so $r_1=Q^T$. Hence $M=N^{-1}$ satisfies $M^{-1}e_1=P$ and $M^Te_1=Q$. If
$\K=\R$ and $n\ge2$, replacing $P_n$ by $-P_n$ if necessary gives
$\det M>0$.

In general, rescaling $P_0$ and $P$ by the same factor we may assume
$Q_0^TP_0=Q^TP=1$. Choose $M_1,M_2$ as above with $M_1^{-1}e_1=P_0$,
$M_1^Te_1=Q_0$, $M_2^{-1}e_1=P$, $M_2^Te_1=Q$, and put $M=M_1^{-1}M_2$. Then
$M^{-1}P_0=M_2^{-1}M_1P_0=M_2^{-1}e_1=P$ and
$M^TQ_0=M_2^TM_1^{-T}Q_0=M_2^Te_1=Q$; moreover $\det M>0$ if
$\det M_1,\det M_2>0$.
\end{proof}

\begin{lemma}\label{lem:orbit}
Let $n\ge1$ and $G_1,G_2\in\cI_n(\K)$ with $\Ups(G_1)=\Ups(G_2)$. Then
$G_2\in\cS_n(\K)\,G_1\,\cS_n(\K)$. If moreover $\K=\R$, $n\ge2$, and either
$G_1,G_2\in\cI_n^+(\R)$ or $G_1,G_2\in\cI_n^-(\R)$, then
$G_2\in\cS_n^+(\R)\,G_1\,\cS_n^+(\R)$.
\end{lemma}

\begin{proof}
Write $G_j=\blk{F_j}{X_j}{Y_j^T}{\eta}$. The vectors $F_1^{-1}X_1$,
$F_2^{-1}X_2$, $Y_1$, $Y_2$ are nonzero and
$Y_1^TF_1^{-1}X_1=Y_2^TF_2^{-1}X_2\neq0$, so by Lemma~\ref{lem:dual} there is
$R\in\GL_n(\K)$ with $R^{-1}F_1^{-1}X_1=F_2^{-1}X_2$ and $R^TY_1=Y_2$. A direct
computation gives
\[
\blk{F_2R^{-1}F_1^{-1}}{0}{0}{1}\blk{F_1}{X_1}{Y_1^T}{\eta}\blk{R}{0}{0}{1}
=\blk{F_2}{F_2R^{-1}F_1^{-1}X_1}{(R^TY_1)^T}{\eta}=G_2 .
\]
In the real case with $\det R>0$ and $\det F_1\det F_2>0$, both outer factors
lie in $\cS_n^+(\R)$.
\end{proof}

\begin{lemma}\label{lem:closure}
Let $n\ge1$, let $\Omega\subseteq\M_{n+1}(\K)$ be closed, let
$\mathcal F\subseteq\GL_n(\K)$ and $L\subseteq\K^*\times\K$. Suppose that every
$G\in\cI_n(\K)$ with $F\in\mathcal F$ and $\Ups(G)\in L$ lies in $\Omega$. Then
every $G\in\cG_n(\K)$ with $F\in\mathcal F$ and $\bUps(G)\in\overline L$ lies
in $\Omega$.
\end{lemma}

\begin{proof}
Let $G=\blk{F}{X}{Y^T}{\eta}$ with $F\in\mathcal F$ and
$\bUps(G)=(a,\eta)=\lim_i(a_i,\eta_i)$, where $(a_i,\eta_i)\in L$; note
$a_i\neq0$.

If $Y\neq0$, the row vector $Y^TF^{-1}$ is nonzero, so we may pick $W\in\K^n$
with $Y^TF^{-1}W=1$. Put $X_i=X+(a_i-a)W$ and
$G_i=\blk{F}{X_i}{Y^T}{\eta_i}$. Then $Y^TF^{-1}X_i=a_i\neq0$, so
$G_i\in\cI_n(\K)$, $\Ups(G_i)=(a_i,\eta_i)\in L$ and thus $G_i\in\Omega$. As
$G_i\to G$ and $\Omega$ is closed, $G\in\Omega$. The case $X\neq0$ is
symmetric. If $X=Y=0$, then $a=0$; fix $U\neq0$ and let
$G'_j=\blk{F}{0}{j^{-1}U^T}{\eta}$. Then $\bUps(G'_j)=(0,\eta)\in\overline L$
and $G'_j$ has nonzero bottom-left block, so $G'_j\in\Omega$ by the first
case; letting $j\to\infty$ gives $G\in\Omega$.
\end{proof}

\begin{remark}
In the original version this lemma was proved by perturbing
$X\pm tW$, $Y+tW$ with $W^TF^{-1}W\neq0$. Such a $W$ need not exist: if
$F^{-1}$ is skew-symmetric (e.g.\ $n=2$, $F=\blk{0}{1}{-1}{0}$), then
$W^TF^{-1}W=0$ for all $W$. Also, $(a_i,\eta_i)$ must be assumed to lie in
$\K^*\times\K$, since $\Ups^{-1}(0,\eta)=\emptyset$.
\end{remark}

\subsection{The product formula and the key lemma}
For $z\in\K$ and $p=(a,\epsilon)$, $q=(b,\delta)\in\K^2$ with $z+ab\neq0$ put
\begin{equation}\label{eq:Phi}
\Phi_z(p,q)=\Big(\frac{(z+a\delta)(z+b\epsilon)}{z+ab},\ z+\epsilon\delta\Big).
\end{equation}

\begin{lemma}\label{lem:product}
Let $n\ge1$, $a,b,r,s\in\K^*$, $\epsilon,\delta\in\K$, and $z=rs$. Put
\[
P=\blk{I_n}{\tfrac ar e_1}{r e_1^T}{\epsilon},\qquad
Q=\blk{I_n}{s e_1}{\tfrac bs e_1^T}{\delta}.
\]
Then $\Ups(P)=(a,\epsilon)$, $\Ups(Q)=(b,\delta)$, the upper-left block of
$PQ$ is $I_n+\frac{ab}{z}e_1e_1^T$, whose determinant is $1+ab/z$, and if
$z\notin\{-ab,-a\delta,-b\epsilon\}$ then $PQ\in\cI_n(\K)$ and
$\Ups(PQ)=\Phi_z\big((a,\epsilon),(b,\delta)\big)$.
\end{lemma}

\begin{proof}
We have
\[
PQ=\blk{I_n+\frac{ab}{z}e_1e_1^T}{\big(s+\frac{a\delta}{r}\big)e_1}
{\big(r+\frac{b\epsilon}{s}\big)e_1^T}{z+\epsilon\delta}.
\]
If $z+ab\neq0$ the upper-left block $F$ is invertible with
$e_1^TF^{-1}e_1=z/(z+ab)$, so
\[
Y^TF^{-1}X=\Big(r+\frac{b\epsilon}{s}\Big)\Big(s+\frac{a\delta}{r}\Big)\frac{z}{z+ab}
=\frac{(z+a\delta)(z+b\epsilon)}{z}\cdot\frac{z}{z+ab},
\]
which is nonzero when $z\neq-a\delta,-b\epsilon$.
\end{proof}

\begin{definition}
A closed set $\widehat L\subseteq\K^2$ is \emph{$\Phi$-closed} if
$\Phi_z(p,q)\in\widehat L$ whenever $p=(a,\epsilon),q=(b,\delta)\in\widehat L$,
$z\in\K^*$ and $z+ab\neq0$.
\end{definition}

\begin{lemma}[Key lemma]\label{lem:key}
Let $\widehat L\subseteq\K^2$ be closed and $\Phi$-closed, and suppose
$(a,\epsilon)\in\widehat L$ with $a\epsilon\neq0$ and $a\neq\epsilon$. Then
$\widehat L=\K^2$.
\end{lemma}

\begin{proof}
\emph{Step 1.} Take $p=q=(a,\epsilon)$ and $z_0=-a\epsilon$. Then $z_0\neq0$
and $z_0+a^2=a(a-\epsilon)\neq0$, so
$\Phi_{z_0}(p,p)=(0,c)\in\widehat L$ with $c=\epsilon(\epsilon-a)\neq0$.

\emph{Step 2.} Take $p=(0,c)$ and $q=(a,\epsilon)$. Here the product of first
coordinates is $0$, so for every $z\neq0$,
$\Phi_z(p,q)=(z+ac,\ z+c\epsilon)\in\widehat L$. Since $\widehat L$ is closed,
the whole line
\[
\ell=\{(t+\kappa,t): t\in\K\},\qquad \kappa=c(a-\epsilon)=-\epsilon(a-\epsilon)^2\neq0,
\]
is contained in $\widehat L$.

\emph{Step 3.} Let $(u,v)\in\K^2$ with $v+\epsilon\kappa-u\neq0$ and set
\[
m=\frac{ua\kappa}{v+\epsilon\kappa-u},\qquad t=\frac{m-v}{a-\epsilon},\qquad z=v-\epsilon t .
\]
With $p=(a,\epsilon)$ and $q=(t+\kappa,t)\in\ell$ one computes
$z+at=m$, $z+\epsilon(t+\kappa)=v+\epsilon\kappa$,
$z+a(t+\kappa)=m+a\kappa$ and $z+\epsilon t=v$, so, provided $z\neq0$ and
$m+a\kappa\neq0$,
\[
\Phi_z(p,q)=\Big(\frac{m(v+\epsilon\kappa)}{m+a\kappa},\ v\Big)=(u,v),
\]
the last equality being equivalent to $m(v+\epsilon\kappa-u)=ua\kappa$.
The conditions $v+\epsilon\kappa-u\neq0$, $m+a\kappa\neq0$ (equivalently
$v+\epsilon\kappa\neq0$) and $z\neq0$ fail only on the zero set of a nonzero
polynomial in $(u,v)$; for the last one, clearing denominators gives
$P(u,v)=(a-\epsilon)v(v+\epsilon\kappa-u)-\epsilon\big(ua\kappa-v(v+\epsilon\kappa-u)\big)$,
and $P(0,v)=av(v+\epsilon\kappa)\not\equiv0$. Hence $\widehat L$ contains a
dense subset of $\K^2$, and being closed it equals $\K^2$.
\end{proof}

\begin{proposition}[Extension principle]\label{prop:extension}
Let $n\ge1$ and let $\Lambda\subseteq\M_{n+1}(\K)$ be a closed subsemigroup
such that
\begin{enumerate}[label=\rm(\roman*)]
\item $\cS_n(\K)\subseteq\Lambda$, and
\item $\Lambda$ contains some $C\in\cI_n(\K)$ with $\Ups(C)=(a,\epsilon)$,
$a\epsilon\neq0$, $a\neq\epsilon$.
\end{enumerate}
Then $\Lambda=\M_{n+1}(\K)$.
\end{proposition}

\begin{proof}
Let $L=\Ups(\Lambda\cap\cI_n(\K))\subseteq\K^*\times\K$. If $G\in\cI_n(\K)$ and
$\Ups(G)\in L$, then $\Ups(G)=\Ups(G_1)$ for some $G_1\in\Lambda\cap\cI_n(\K)$,
and by Lemma~\ref{lem:orbit} and (i), $G\in\cS_n G_1\cS_n\subseteq\Lambda$. By
Lemma~\ref{lem:closure} (with $\mathcal F=\GL_n(\K)$, $\Omega=\Lambda$),
\begin{equation}\label{eq:barL}
\{G\in\cG_n(\K):\bUps(G)\in\overline L\}\subseteq\Lambda .
\end{equation}
We claim $\overline L$ is $\Phi$-closed. Let $p,q\in\overline L$, $z\neq0$,
$z+ab\neq0$. Choose $p_j=(a_j,\epsilon_j)\to p$ and $q_j=(b_j,\delta_j)\to q$
in $L$, and $z_j\to z$ with
$z_j\notin\{0,-a_jb_j,-a_j\delta_j,-b_j\epsilon_j\}$. With $r=1$, $s=z_j$, the
matrices $P,Q$ of Lemma~\ref{lem:product} lie in $\Lambda$ (they are in
$\cI_n(\K)$ with $\Ups$-values in $L$), so $PQ\in\Lambda\cap\cI_n(\K)$ and
$\Phi_{z_j}(p_j,q_j)\in L$. Since $\Phi$ is continuous where $z+ab\neq0$,
$\Phi_z(p,q)\in\overline L$. By (ii) and Lemma~\ref{lem:key},
$\overline L=\K^2$, so $\cG_n(\K)\subseteq\Lambda$ by \eqref{eq:barL}. As
$\cG_n(\K)$ is dense and $\Lambda$ is closed, $\Lambda=\M_{n+1}(\K)$.
\end{proof}

\section{The complex case}\label{sec:complex}

For $n\ge2$ let $\cC_n$ be the set of $A\in\M_n(\C)$ that are similar to a
matrix of the form $\diag(A',1)$ with $A'\in\GL_{n-1}(\C)$; equivalently,
$A$ is invertible and $1$ is an eigenvalue of $A$ having a Jordan block of
size $1$.

\begin{remark}
In the original version $\cC_n$ consisted of diagonalizable matrices
$\diag(Z_1,\dots,Z_k)$ with $Z_k=1$ and $Z_2,\dots,Z_{k-1}$ roots of unity.
The roots-of-unity condition plays no role in the argument, while the
invertibility of $Z_1$ is needed (otherwise the matrix $C$ in the proof of
Lemma~\ref{lem:complex-step} is not in $\cI_n(\C)$). Every matrix in the
original class lies in the present one.
\end{remark}

\begin{lemma}\label{lem:complex-step}
Let $n\ge2$ and $A,E\in\M_n(\C)$ with $A\in\cC_n$ and $\gen{A,E}$ dense in
$\M_n(\C)$. Then there exist $C,D\in\M_{n+1}(\C)$ with $C\in\cC_{n+1}$ and
$\gen{C,D}$ dense in $\M_{n+1}(\C)$.
\end{lemma}

\begin{proof}
After a change of basis, $A=\diag(A',1)$ with $A'\in\GL_{n-1}(\C)$. Let
$B'\in\GL_{n-1}(\C)$ be a square root of $A'$ (e.g.\ $B'=e^{X/2}$ where
$A'=e^X$), let $\alpha=\sqrt2/2$, and define
\begin{equation}\label{eq:CDcomplex}
C=\begin{pmatrix} B' & 0\\ 0 & Z\end{pmatrix},\quad
Z=\blk{\alpha}{\alpha}{\alpha}{-\alpha},\qquad
D=\blk{E}{0}{0}{1}.
\end{equation}
Since $Z^2=I_2$, we have $C^2=\diag(A,1)$. Let $\Lambda$ be the closure of
$\gen{C,D}$. Every word in $C^2$ and $D$ has the form $\diag(w(A,E),1)$, and
$\gen{A,E}$ is dense, so $\diag(M,1)\in\Lambda$ for all $M\in\M_n(\C)$; in
particular $\cS_n(\C)\subseteq\Lambda$. In the block form \eqref{eq:block},
$C$ has $F=\diag(B',\alpha)$, $X=Y=\alpha e_n$ and $\eta=-\alpha$, so
$C\in\cI_n(\C)$ and
\[
\Ups(C)=\big(\alpha^2\cdot\alpha^{-1},-\alpha\big)=(\alpha,-\alpha).
\]
Since $\alpha\cdot(-\alpha)\neq0$ and $\alpha\neq-\alpha$,
Proposition~\ref{prop:extension} gives $\Lambda=\M_{n+1}(\C)$. Finally, $Z$
is a symmetric orthogonal matrix with eigenvalues $\pm1$, so $C$ is similar
to $\diag(B',-1,1)$, i.e.\ $C\in\cC_{n+1}$.
\end{proof}

\begin{theorem}\label{thm:complex}
For every $n\ge1$ there is a pair of matrices in $\M_n(\C)$ generating a
dense subsemigroup of $\M_n(\C)$. For $n\ge2$ one of the two matrices can be
chosen in $\cC_n$.
\end{theorem}

\begin{proof}
For $n=1$ use Lemma~\ref{lem:dim1}(b). For $n=2$, by
\cite[Example~13]{J-2x2} there are $u,a,b\in\C^*$ such that $\blk{u}{a}{1}{0}$
and $\blk{1}{0}{0}{b}$ generate a dense subsemigroup of $\M_2(\C)$; the second
matrix lies in $\cC_2$. The induction step is
Lemma~\ref{lem:complex-step}.
\end{proof}

\section{The real case}\label{sec:real}

\subsection{A criterion in $\SL_2(\R)$ and the case $n=2$}
For $\theta\in\R$ let $R_\theta=\blk{\cos\theta}{-\sin\theta}{\sin\theta}{\cos\theta}$.
Recall that an element $h\in\SL_2(\R)$ with $|\tr h|<2$ is conjugate to some
$R_\theta$, and that $h$ has infinite order if and only if $\theta/\pi\notin\Q$.
If $\theta/\pi\in\Q$, then $\tr h=2\cos\theta$ is a sum of two roots of unity,
hence an algebraic integer.

\begin{lemma}\label{lem:SL2}
Let $T\subseteq\SL_2(\R)$ be a closed subsemigroup containing an element $h$
with $|\tr h|<2$ of infinite order, and an element $g$ that does not lie in
the closure $K$ of $\{h^m:m\ge1\}$. Then $T=\SL_2(\R)$.
\end{lemma}

\begin{proof}
Conjugating, we may assume $h=R_\theta$ with $\theta/\pi\notin\Q$. Then
$\{R_{m\theta}:m\ge1\}$ is dense in $\SO(2)$, so $K=\SO(2)\subseteq T$. Let
$a_t=\diag(e^t,e^{-t})$. By the Cartan decomposition every $x\in\SL_2(\R)$
can be written $x=k_1a_tk_2$ with $k_1,k_2\in\SO(2)$ and a unique $t=t(x)\ge0$,
namely $e^{t(x)}=\|x\|$ (operator norm); $t(x)=0$ iff $x\in\SO(2)$. Since
$K\subseteq T$ is a group, $a_{t(x)}\in KxK\subseteq T$ for every $x\in T$. Put
$t=t(g)>0$ and $\phi(s)=a_tR_sa_t\in T$ for $s\in[0,\pi/2]$. Then
$\phi(0)=a_{2t}$ and, since $R_{\pi/2}^{-1}a_tR_{\pi/2}=a_{-t}$,
$\phi(\pi/2)=R_{\pi/2}$. As $s\mapsto t(\phi(s))$ is continuous, it takes every
value $\tau\in[0,2t]$, so $a_\tau\in T$ for all $\tau\in[0,2t]$, hence for all
$\tau\ge0$ because $a_\tau a_{\tau'}=a_{\tau+\tau'}$. Therefore
$T\supseteq\SO(2)\{a_\tau:\tau\ge0\}\SO(2)=\SL_2(\R)$.
\end{proof}

\begin{lemma}\label{lem:GL2plus}
Let
\begin{equation}\label{eq:abe}
a=-2^{3/5},\qquad b=\tfrac83,\qquad e=-2^{-4/5},
\end{equation}
and $A=\blk{a}{e}{1}{0}$, $B=\blk{1}{0}{0}{b}$. Then $\gen{A,B}$ is dense in
$\{M\in\M_2(\R):\det M\ge0\}$, i.e.\ in $\GL_2^+(\R)$.
\end{lemma}

\begin{proof}
Let $T$ be the closure of $\gen{A,B}$. Note $\det A=-e>0$ and $\det B>0$, so
$T\subseteq\{\det\ge0\}$. A direct computation gives
\[
C:=ABA^3BA=\blk{4/9}{0}{0}{1}\in T .
\]
\emph{Scalars.} Let $c=4/9$. Since $c<1<b$ and
$\ln c/\ln b\notin\Q$ (because $2^{2p}3^{-2p}=2^{3q}3^{-q}$ forces $p=q=0$),
for every $d<0$ there are $k_i,l_i\to\infty$ with $b^{k_i}c^{l_i}\to d/e>0$.
Then
\[
\blk{0}{d}{1}{0}=\lim_{i\to\infty}\blk{c^{l_i}a}{c^{l_i}b^{k_i}e}{1}{0}
=\lim_{i\to\infty}\blk{c^{l_i}}{0}{0}{1}A\blk{1}{0}{0}{b^{k_i}}\in T,
\]
so $dI_2=\blk{0}{d}{1}{0}^2\in T$. Hence $dI_2\in T$ for all $d<0$, and then
for all $d>0$ (as $(-\sqrt d\,I_2)^2$) and $d=0$. Consequently, if $N\in T$
and $\det N>0$, then $(\det N)^{-1/2}N\in T\cap\SL_2(\R)$.

\emph{An elliptic element.} $AB=\blk{a}{eb}{1}{0}$ has $\tr(AB)=a$ and
$\det(AB)=-eb=\frac83 2^{-4/5}$, so $h=(\det AB)^{-1/2}AB\in T\cap\SL_2(\R)$ has
\[
\tr h=\frac{-2^{3/5}}{2^{-2/5}\sqrt{8/3}}=-\sqrt{3/2}.
\]
Thus $|\tr h|<2$, and $\tr h$ is not an algebraic integer (its square is
$3/2$), so $h$ has infinite order. The element $g=b^{-1/2}B=\diag(b^{-1/2},b^{1/2})$
lies in $T\cap\SL_2(\R)$ and in no compact subgroup. By
Lemma~\ref{lem:SL2}, $T\supseteq\SL_2(\R)$, and since $T$ contains all
positive scalars, $T\supseteq\GL_2^+(\R)$.
\end{proof}

\begin{remark}
In the original version the last part of this proof was deduced from
\cite[Prop.~4.1]{J-2x2} on linear fractional maps. The argument above is
self-contained. The same argument also yields a short proof of the real
$2\times2$ example of \cite{J-2x2}:
\end{remark}

\begin{proposition}[{\cite[Example~11]{J-2x2}}]\label{prop:M2R}
The matrices $A=\blk{1}{1/2}{1}{0}$ and $B=\blk{1}{0}{0}{-8/3}$ generate a
dense subsemigroup of $\M_2(\R)$.
\end{proposition}

\begin{proof}
Let $T$ be the closure of $\gen{A,B}$. One checks $ABA^3BA=\diag(c,1)$ with
$c=-2/9$. Put $b=-8/3$. Since $\ln|c|<0<\ln|b|$ and $\ln|c|/\ln|b|\notin\Q$,
the numbers $|c|^{l}|b|^{k}$ with $l+k$ of prescribed parity and $l\ge l_0$
are dense in $(0,\infty)$ for every $l_0$ (as in Lemma~\ref{lem:dim1}(a)).
Hence $c^lb^k$, $l\to\infty$, approximates every real number, and as in the
proof of Lemma~\ref{lem:GL2plus},
$\diag(c^l,1)A\diag(1,b^k)=\blk{c^l}{\frac12c^lb^k}{1}{0}$ converges to
$\blk{0}{d}{1}{0}$ for any prescribed $d\in\R$. Therefore $dI_2\in T$ for all
$d\in\R$. Next, $AB=\blk{1}{-4/3}{1}{0}$ has $\det AB=4/3$, so
$h=(4/3)^{-1/2}AB$ has $\tr h=\sqrt3/2$, which is not an algebraic integer;
thus $h$ is elliptic of infinite order. Also $(8/3)^{-1}B^2=\diag(3/8,8/3)\in T$
is not in any compact subgroup. By Lemma~\ref{lem:SL2},
$T\supseteq\SL_2(\R)$, hence $T\supseteq\GL_2^+(\R)$. Since $\det B<0$,
$T\supseteq B\,\GL_2^+(\R)=\GL_2^-(\R)$, and so $T=\M_2(\R)$.
\end{proof}

\subsection{The case $n=3$}
\begin{lemma}\label{lem:n3}
Let $a,b,e$ be as in \eqref{eq:abe} and $\alpha=\sqrt2/2$. Then the matrices
\[
A=\begin{pmatrix}\sqrt b&0&0\\0&\alpha&\alpha\\0&\alpha&-\alpha\end{pmatrix},\qquad
E=\begin{pmatrix}0&1&0\\e&a&0\\0&0&1\end{pmatrix}
\]
generate a dense subsemigroup of $\M_3(\R)$.
\end{lemma}

\begin{proof}
Let $\Lambda$ be the closure of $\gen{A,E}$. We have $A^2=\diag(b,1,1)$.
Conjugating Lemma~\ref{lem:GL2plus} by $\blk{0}{1}{1}{0}$ shows that
$\diag(b,1)$ and $\blk{0}{1}{e}{a}$ generate a dense subsemigroup of
$\GL_2^+(\R)$; hence $\cS_2^+(\R)\subseteq\Lambda$. Let
$L^\pm=\Ups(\Lambda\cap\cI_2^\pm(\R))$. By Lemma~\ref{lem:orbit} (second part)
and Lemma~\ref{lem:closure} with $\mathcal F=\GL_2^\pm(\R)$,
\begin{equation}\label{eq:Lpm}
\{G\in\cG_2(\R):\ \pm\det F>0,\ \bUps(G)\in\overline{L^\pm}\}\subseteq\Lambda.
\end{equation}

\emph{The set $\overline{L^+}$.} By Lemma~\ref{lem:product}, if
$p=(a_1,\epsilon),q=(b_1,\delta)\in L^+$ and $z\neq0$ satisfies
$1+a_1b_1/z>0$ and $z\notin\{-a_1\delta,-b_1\epsilon\}$, then
$\Phi_z(p,q)\in L^+$ (the matrices $P,Q$ of that lemma lie in
$\Lambda\cap\cI_2^+$ by \eqref{eq:Lpm}, and the upper-left block of $PQ$ has
determinant $1+a_1b_1/z>0$). Arguing as in
Proposition~\ref{prop:extension}, the same holds for $p,q\in\overline{L^+}$
whenever $z\neq0$ and $1+a_1b_1/z>0$. Now $A\in\cI_2^+(\R)$ with
$F=\diag(\sqrt b,\alpha)$, so $(\alpha,-\alpha)=\Ups(A)\in L^+$.

(i) For $x<-1/2$ we have $1+\alpha^2/x>0$, so
$p_x:=\Phi_x((\alpha,-\alpha),(\alpha,-\alpha))=(A_x,\delta_x)\in\overline{L^+}$,
where $A_x=(x-\frac12)^2/(x+\frac12)$ and $\delta_x=x+\frac12$. For $z>0$,
$1+A_x^2/z>0$, hence
\[
\Phi_z(p_x,p_x)=\Big(\frac{(z+A_x\delta_x)^2}{z+A_x^2},\ z+\delta_x^2\Big)\in\overline{L^+}.
\]
As $x\to-\frac12^-$ we have $A_x\to-\infty$, $\delta_x\to0$ and
$A_x\delta_x=(x-\frac12)^2\to1$, so $(0,z)\in\overline{L^+}$ for all $z>0$.

(ii) For $w_1,w_2>0$ and $z\neq0$,
$\Phi_z((0,w_1),(0,w_2))=(z,z+w_1w_2)\in\overline{L^+}$; hence
$\{(u,v):v>u\}\subseteq\overline{L^+}$.

(iii) For $w>0$ and $z\neq0$,
$\Phi_z((0,w),(\alpha,-\alpha))=(z+\alpha w,z-\alpha w)\in\overline{L^+}$;
hence $\{(u,v):u>v\}\subseteq\overline{L^+}$.

Thus $\overline{L^+}=\R^2$ and, by \eqref{eq:Lpm}, every $G\in\M_3(\R)$ in
block form with $\det F>0$ lies in $\Lambda$.

\emph{The set $\overline{L^-}$.} For $c,d,v\in\R$ with $c\neq0$,
\[
\begin{pmatrix}1&0&1\\0&1&0\\ \frac{v-1}{c}&0&1\end{pmatrix}
\begin{pmatrix}1&0&c\\0&1&0\\ d&0&1\end{pmatrix}
=\begin{pmatrix}1+d&0&c+1\\0&1&0\\ d+\frac{v-1}{c}&0&v\end{pmatrix}\in\Lambda,
\]
because both factors have $F=I_2$. The product has $F=\diag(1+d,1)$, so it
has $\bUps$-value $(u,v)$ with
$u=(c+1)\big(d+\frac{v-1}{c}\big)/(1+d)$, and it lies in $\cI_2^-(\R)$ when
$d<-1$ and $u\neq0$. Solving for $d$,
\begin{equation}\label{eq:d}
d=\frac{c(u-v+1)+1-v}{c(c+1-u)} .
\end{equation}
(The numerator was misprinted as $c(u-v-1)+1-v$ in the original version.)
Given $u\notin\{0,1\}$ and $v\neq1$, letting $c\to0$ with the sign of $c$
chosen so that $(1-v)/(c(1-u))<0$ gives $d\to-\infty$, so $d<-1$ for $|c|$
small. Hence $(u,v)\in L^-$ for all such $(u,v)$, and $\overline{L^-}=\R^2$.
By \eqref{eq:Lpm}, all $G$ with $\det F<0$ lie in $\Lambda$. Thus
$\cG_2(\R)\subseteq\Lambda$ and $\Lambda=\M_3(\R)$.
\end{proof}

\subsection{The induction}
For $n\ge3$ let
\begin{align*}
\cR_n^+&=\{A\in\M_n(\R): A\sim\diag(e^\Theta,1)\ \text{for some}\ \Theta\in\M_{n-1}(\R)\},\\
\cR_n^-&=\{A\in\M_n(\R): A\sim\diag(e^\Theta,1,-1)\ \text{for some}\ \Theta\in\M_{n-2}(\R)\},
\end{align*}
where $\sim$ denotes similarity over $\R$, and $\cR_n=\cR_n^+\cup\cR_n^-$. Note
that $\det A>0$ for $A\in\cR_n^+$ and $\det A<0$ for $A\in\cR_n^-$. The matrix
$A$ of Lemma~\ref{lem:n3} is similar to $\diag(\sqrt b,1,-1)$, hence lies in
$\cR_3^-$.

\begin{remark}
The class $\cR_n$ of the original version (block-diagonal matrices
$\diag(Z_1,\dots,Z_k)$ with $Z_1>0$, $Z_k=\sgn\det A$, intermediate blocks
equal to $1$ or to $R_{2^{-m}\pi}$, and $1$ an eigenvalue) is not quite
stable under the induction as stated: when $\det A<0$ the eigenvalue $1$ must
occur among the intermediate blocks (not only as $Z_1$), for otherwise the
matrix $C$ constructed in Case~2 below cannot be put in the required form.
The classes $\cR_n^\pm$ above are stable and contain all matrices actually
produced by the original construction.
\end{remark}

\begin{lemma}\label{lem:sign}
Let $n\ge1$ and let $\Lambda\subseteq\M_{n+1}(\R)$ be a closed subsemigroup
such that
\begin{enumerate}[label=\rm(\roman*)]
\item for every $M\in\GL_n(\R)$ there is $\sigma\in\{\pm1\}$ with
$\diag(M,\sigma)\in\Lambda$, and
\item $\diag(E,-\sgn\det E)\in\Lambda$ for some $E\in\GL_n(\R)$.
\end{enumerate}
Then $\cS_n(\R)\subseteq\Lambda$.
\end{lemma}

\begin{proof}
For $M\in\GL_n(\R)$ let $\Sigma(M)=\{\sigma:\diag(M,\sigma)\in\Lambda\}\neq\emptyset$.
If $\Sigma(M_0)=\{\pm1\}$ for some $M_0$, then for every $N$, writing
$\diag(N,\pm\sigma_1)=\diag(M_0,\pm1)\diag(M_0^{-1}N,\sigma_1)$ with
$\sigma_1\in\Sigma(M_0^{-1}N)$ shows $\Sigma(N)=\{\pm1\}$, and we are done.
Otherwise $\Sigma(M)=\{\sigma(M)\}$ defines a map
$\sigma:\GL_n(\R)\to\{\pm1\}$, which is a homomorphism because $\Lambda$ is a
semigroup. Since $\Lambda$ is closed, each fibre $\sigma^{-1}(\pm1)$ is closed
in $\GL_n(\R)$; the two fibres partition $\GL_n(\R)$, so both are open and
$\sigma$ is constant on the two components $\GL_n^\pm(\R)$. As $\sigma(I_n)=1$,
either $\sigma\equiv1$, which gives $\cS_n(\R)\subseteq\Lambda$, or
$\sigma=\sgn\det$, which contradicts (ii).
\end{proof}

\begin{remark}
In the original version, the analogue of Lemma~\ref{lem:sign} was invoked
(via the Jordan--Dickson theorem) only in the case $\det A<0$. It is needed
in both cases: when $\det A>0$, the last diagonal entry of a word in $C^2$
and $D$ equals $(-\sgn\det E)^{\#D}$, which is not identically $1$.
\end{remark}

\begin{lemma}\label{lem:real-step}
Let $n\ge3$ and $A,E\in\M_n(\R)$ with $A\in\cR_n$ and $\gen{A,E}$ dense in
$\M_n(\R)$. Then there exist $C,D\in\M_{n+1}(\R)$ with $C\in\cR_{n+1}$ and
$\gen{C,D}$ dense in $\M_{n+1}(\R)$. More precisely, $A\in\cR_n^\pm$ implies
$C\in\cR_{n+1}^\mp$.
\end{lemma}

\begin{proof}
By Lemma~\ref{lem:invertible}, $E$ is invertible. Let
$D=\diag(E,-\sgn\det E)$ and let $\Lambda$ be the closure of $\gen{C,D}$,
where $C$ is defined below after a change of basis. In both cases $C$
satisfies $C^2=\diag(A,\sgn\det A)$.

\emph{Case 1: $A\in\cR_n^+$.} Assume $A=\diag(e^\Theta,1)$ with
$\Theta\in\M_{n-1}(\R)$ and let $C=\diag(e^{\Theta/2},Z)$ with
$Z=\blk{\alpha}{\alpha}{\alpha}{-\alpha}$ acting on the last two
coordinates. Then $C^2=\diag(A,1)$; in block form $C$ has
$F=\diag(e^{\Theta/2},\alpha)$, $X=Y=\alpha e_n$, $\eta=-\alpha$, so
$\Ups(C)=(\alpha,-\alpha)$. Moreover
$C\sim\diag(e^{\Theta/2},1,-1)\in\cR_{n+1}^-$.

\emph{Case 2: $A\in\cR_n^-$.} Assume $A=\diag(e^\Theta,1,-1)$ with $\Theta\in\M_{n-2}(\R)$ and let
\[
C=\diag(e^{\Theta/2},1,J),\qquad J=\blk{1}{-\sqrt2}{\sqrt2}{-1},
\]
with $J$ acting on the last two coordinates. Since $J^2=-I_2$, $C^2=\diag(A,-1)$.
In block form $C$ has $F=\diag(e^{\Theta/2},1,1)$, $X=-\sqrt2\,e_n$,
$Y=\sqrt2\,e_n$, $\eta=-1$, so $\Ups(C)=(-2,-1)$. Since $J$ has eigenvalues
$\pm i$, we have $J\sim R_{\pi/2}=e^{(\pi/2)\Omega}$ with
$\Omega=\blk{0}{-1}{1}{0}$, and therefore
\[
C\sim\diag\big(\exp\diag(\Theta/2,\tfrac\pi2\Omega),\,1\big)\in\cR_{n+1}^+ .
\]

In both cases, every word in $C^2$ and $D$ has the form
$\diag(w(A,E),\pm1)$. Given $M\in\M_n(\R)$, choose words with
$w_j(A,E)\to M$ and pass to a subsequence along which the last entry is
constant; thus hypothesis (i) of Lemma~\ref{lem:sign} holds, and (ii) holds
since $D\in\Lambda$. Hence $\cS_n(\R)\subseteq\Lambda$. Finally,
$\Ups(C)=(a,\epsilon)$ satisfies $a\epsilon\neq0$ and $a\neq\epsilon$ in both
cases, so Proposition~\ref{prop:extension} gives $\Lambda=\M_{n+1}(\R)$.
\end{proof}

\begin{remark}
In Case~2 the value $\Ups(C)=(-2,-1)$ differs from the value
$(\alpha,-\alpha)$ used in Step~3 of the original proof of the complex case,
so the original statement that ``the rest of the proof is the same'' required
an additional argument. Lemma~\ref{lem:key} supplies it uniformly.
\end{remark}

\begin{theorem}\label{thm:real}
For every $n\ge1$ there is a pair of matrices in $\M_n(\R)$ generating a
dense subsemigroup of $\M_n(\R)$. For $n\ge3$ one of the two matrices can be
chosen in $\cR_n$. Moreover, the construction is completely explicit.
\end{theorem}

\begin{proof}
For $n=1$ use Lemma~\ref{lem:dim1}(a), for $n=2$
Proposition~\ref{prop:M2R}, and for $n=3$ Lemma~\ref{lem:n3}. The induction
step is Lemma~\ref{lem:real-step}. Explicitness follows because every change
of basis used in Lemma~\ref{lem:real-step} only diagonalizes the explicit
$2\times2$ blocks $Z$ and $J$ and permutes coordinates.
\end{proof}

\section{Topologically $n$-transitive semigroups}\label{sec:transitive}

We identify $(\K^n)^k$ with $\M_{n\times k}(\K)$, a $k$-tuple of vectors
being the matrix with these columns; a semigroup $G\subseteq\M_n(\K)$ then
acts by left multiplication.

\begin{proposition}\label{prop:not-n+1}
Let $G\subseteq\M_n(\K)$ be any semigroup (for instance $\M_n(\K)$ itself).
Then the action of $G$ on $(\K^n)^k$ is not topologically transitive for any
$k\ge n+1$.
\end{proposition}

\begin{proof}
If the action on $(\K^n)^k$ is topologically transitive, so is the action on
$(\K^n)^{n+1}$, since the projection onto the first $n+1$ factors is
equivariant, continuous, open and surjective. Let $W\subseteq(\K^n)^{n+1}$ be
the open set of tuples whose first $n$ vectors are linearly independent, and
let $\gamma:W\to\K^n$ be the continuous map with
$X_{n+1}=\sum_{i\le n}\gamma_i(X)X_i$. If $X\in W$, $g\in\M_n(\K)$ and $gX\in W$,
then applying $g$ to this relation gives $\gamma(gX)=\gamma(X)$. Since $\gamma$ is
surjective, $U=\gamma^{-1}(B(0,1))$ and $V=\gamma^{-1}(B(3e_1,1))$ are nonempty
open sets, and $gU\cap V=\emptyset$ for every $g\in\M_n(\K)$.
\end{proof}

\begin{remark}
The original proof of this proposition showed only that no orbit in
$(\K^n)^{n+1}$ is dense. For uncountable semigroups (such as $\M_n(\K)$)
the absence of dense orbits does not formally imply the absence of
topological transitivity, so the direct argument above is needed.
\end{remark}

\begin{proposition}\label{prop:equiv}
For a subsemigroup $G\subseteq\M_n(\K)$ the following are equivalent:
\begin{enumerate}[label=\rm(\alph*)]
\item $G$ is dense in $\M_n(\K)$;
\item the action of $G$ on $\K^n$ is topologically $n$-transitive;
\item the action of $G$ on $(\K^n)^n$ has a dense orbit.
\end{enumerate}
\end{proposition}

\begin{proof}
(a)$\Rightarrow$(c): the orbit of $I_n$ is $G$.
(c)$\Rightarrow$(a): if $GX$ is dense, then $X$ is invertible (otherwise
$GX$ lies in the proper closed subspace $\M_n(\K)X$), and right multiplication
by $X^{-1}$ is a homeomorphism, so $G=(GX)X^{-1}$ is dense.
(a)$\Rightarrow$(b): given nonempty open $U,V$, pick an invertible $M\in U$;
then $GM$ is dense, so $gM\in V$ for some $g\in G$.
(b)$\Rightarrow$(a): let $\mathcal W\subseteq\M_n(\K)$ be nonempty and open,
and pick $P\in\mathcal W\cap\GL_n(\K)$. By continuity of
$(Q,M)\mapsto QM^{-1}$ there are open neighborhoods $V$ of $P$ and $U$ of
$I_n$, with $U\subseteq\GL_n(\K)$, such that $VU^{-1}\subseteq\mathcal W$.
By (b) there are $g\in G$ and $M\in U$ with $gM\in V$, so
$g=(gM)M^{-1}\in\mathcal W$.
\end{proof}

\begin{theorem}\label{thm:orbit-closures}
Let $G\subseteq\M_n(\K)$ be a dense subsemigroup, $k\ge1$ and
$X\in\M_{n\times k}(\K)$. Then
\[
\overline{GX}=\M_n(\K)X=\{Y\in\M_{n\times k}(\K):\ \ker X\subseteq\ker Y\}.
\]
This is a linear subspace of dimension $n\cdot\rank X$, and it contains the
orbit closure of $X$ under every semigroup of linear maps of $\K^n$.
\end{theorem}

\begin{proof}
If $Y=FX$, then $\ker X\subseteq\ker Y$. Conversely, if $\ker X\subseteq\ker Y$,
then the row space of $Y$ is contained in that of $X$, so $Y=FX$ for some $F$.
The map $F\mapsto FX$ is continuous and maps $\M_n(\K)$ onto $\M_n(\K)X$; since
$G$ is dense in $\M_n(\K)$, $GX$ is dense in $\M_n(\K)X$, which is closed.
Every linear semigroup maps $X$ into $\M_n(\K)X$. The dimension count follows
from $\M_n(\K)X\cong\M_{n\times r}(\K)$, $r=\rank X$.
\end{proof}

\begin{corollary}\label{cor:transitive}
Let $A,E$ be as in Theorem~\ref{thm:complex} ($\K=\C$) or
Theorem~\ref{thm:real} ($\K=\R$). Then for $1\le k\le n$ the action of
$\gen{A,E}$ on $\K^n$ is topologically $k$-transitive, and a $k$-tuple has a
dense orbit if and only if it is linearly independent. For $k>n$ the action
is not topologically $k$-transitive, and the orbit of every $k$-tuple is
contained in a proper linear subspace.
\end{corollary}

\begin{proof}
By Theorem~\ref{thm:orbit-closures}, the orbit of $X\in(\K^n)^k$ is dense iff
$n\cdot\rank X=nk$, i.e.\ iff the columns of $X$ are independent. For $k\le n$
such tuples form an open dense set, which gives topological $k$-transitivity
(pick $X$ in the given open set $U$ with dense orbit). For $k>n$ use
Proposition~\ref{prop:not-n+1}.
\end{proof}

\section{Remarks and questions}\label{sec:remarks}

\begin{enumerate}[label=\arabic*., leftmargin=*]
\item \emph{Optimality.} Two generators is the least possible: the powers of
a single matrix are never dense \cite{Rolewicz}, and by
Proposition~\ref{prop:not-n+1} the degree $n$ of transitivity in
Corollary~\ref{cor:transitive} cannot be improved by any semigroup of linear
maps.

\item \emph{Comparison with groups.} Abels and Vinberg \cite{AbelsVinberg}
showed that connected semisimple Lie groups with finite center contain dense
$2$-generator subsemigroups. Theorem~A and Lemma~\ref{lem:invertible} give
dense $2$-generator subsemigroups of the non-semisimple groups $\GL_n(\K)$,
by an elementary and explicit construction.

\item \emph{The positive cone.} Lemma~\ref{lem:GL2plus} shows that
$\GL_2^+(\R)$ has a dense $2$-generator subsemigroup.
\end{enumerate}

\begin{question}
Does $\GL_n^+(\R)$ have a dense $2$-generator subsemigroup for every $n\ge3$?
(The construction of Lemma~\ref{lem:real-step} necessarily produces a
generator of negative determinant in every other dimension.)
\end{question}

\begin{question}
Is the set of pairs $(A,B)\in\M_n(\K)^2$ that generate a dense subsemigroup
of $\M_n(\K)$ dense in $\M_n(\K)^2$? Does it have positive Lebesgue measure?
\end{question}

\begin{question}
For which natural closed subsemigroups $\mathcal T\subseteq\M_n(\K)$ (for
example, triangular matrices or matrices with nonnegative entries) is the
minimal number of generators of a dense subsemigroup of $\mathcal T$ known?
For $2\times2$ matrices with nonnegative entries this number is $3$
\cite{J-2x2}, and for $n\times n$ lower triangular matrices it is $n+1$
\cite{J-triangular}.
\end{question}


\begin{thebibliography}{99}

\bibitem{AbelsManoussos}
H.~Abels and A.~Manoussos, Topological generators of abelian Lie groups and
hypercyclic finitely generated abelian semigroups of matrices,
\emph{Adv. Math.} \textbf{229} (2012), 1862--1872.

\bibitem{AbelsVinberg}
H.~Abels and E.~B.~Vinberg, Generating semisimple groups by tori,
\emph{J. Algebra} \textbf{328} (2011), 114--121.

\bibitem{Apostol}
T.~M.~Apostol, \emph{Modular Functions and Dirichlet Series in Number
Theory}, 2nd ed., Springer, 1990.

\bibitem{Ayadi}
A.~Ayadi, Hypercyclic abelian semigroup of matrices on $\C^n$ and $\R^n$ and
$k$-transitivity ($k\ge2$), \emph{Appl. Gen. Topol.} \textbf{12} (2011),
no.~1, 35--39.

\bibitem{AyadiMarzougui}
A.~Ayadi and H.~Marzougui, Hypercyclic abelian semigroups of matrices on
$\R^n$, preprint, arXiv:1010.5915.

\bibitem{CHM}
G.~Costakis, D.~Hadjiloucas and A.~Manoussos, Dynamics of tuples of
matrices, \emph{Proc. Amer. Math. Soc.} \textbf{137} (2009), 1025--1034.

\bibitem{Feldman}
N.~S.~Feldman, Hypercyclic tuples of operators and somewhere dense orbits,
\emph{J. Math. Anal. Appl.} \textbf{346} (2008), 82--98.

\bibitem{J-2x2}
M.~Javaheri, Dense $2$-generator subsemigroups of $2\times2$ matrices,
\emph{J. Math. Anal. Appl.} \textbf{387} (2012), 103--113.

\bibitem{J-orbits}
M.~Javaheri, Semigroups of matrices with dense orbits, \emph{Dyn. Syst.}
\textbf{26} (2011), no.~3, 235--243.

\bibitem{J-triangular}
M.~Javaheri, Dense semigroups of triangular matrices, preprint,
arXiv:1707.06293.

\bibitem{Rolewicz}
S.~Rolewicz, On orbits of elements, \emph{Studia Math.} \textbf{32} (1969),
17--22.

\end{thebibliography}
\end{document}